\documentclass[11pt]{amsart}

\pdfoutput=1

\usepackage{amsmath}
\usepackage{fullpage}
\usepackage{xspace}
\usepackage[psamsfonts]{amssymb}
\usepackage[latin1]{inputenc}
\usepackage{graphicx,color}
\usepackage[curve]{xypic} 
\usepackage{hyperref}
\usepackage{graphicx}

\usepackage{amsmath}%
\usepackage{amsthm}%
\usepackage{amscd}
\usepackage{amsfonts}%
\usepackage{amssymb}%
\usepackage{graphicx}

\usepackage{mathrsfs}

\usepackage{tikz}
\usetikzlibrary{matrix,arrows}

\usepackage{tikz-cd}

\newtheorem{theorem}{Theorem}[section]

\newtheorem{conjecture}[theorem]{Conjecture}
\newtheorem{corollary}[theorem]{Corollary}

\newtheorem{lemma}[theorem]{Lemma}

\theoremstyle{remark}

\numberwithin{equation}{section}

\newcommand{\pfrak}{\mathfrak{p}}

\newcommand{\Pcal}{\mathscr{P}}

\newcommand{\Z}{\mathbb{Z}}
\newcommand{\C}{\mathbb{C}}

\newcommand{\Q}{\mathbb{Q}}

\newcommand{\rad}{\mathrm{rad}}

\newcommand{\Norm}{\mathrm{Norm}}

\makeatletter
\@namedef{subjclassname@2020}{%
  \textup{2020} Mathematics Subject Classification}
\makeatother

  \DeclareFontFamily{U}{wncy}{}
    \DeclareFontShape{U}{wncy}{m}{n}{<->wncyr10}{}
    \DeclareSymbolFont{mcy}{U}{wncy}{m}{n}
    \DeclareMathSymbol{\Sha}{\mathord}{mcy}{"58}

\begin{document}
\title[]{The largest prime factor of $n^2+1$ and improvements on subexponential $ABC$}

\author{Hector Pasten}
\address{ Departamento de Matem\'aticas,
Pontificia Universidad Cat\'olica de Chile.
Facultad de Matem\'aticas,
4860 Av.\ Vicu\~na Mackenna,
Macul, RM, Chile}
\email[H. Pasten]{hector.pasten@uc.cl}%

\thanks{Supported by ANID Fondecyt Regular grant 1230507 from Chile.}
\date{\today}
\dedicatory{Dedicado a la memoria de mi padre, quien siempre me apoy\'o en todo.}

\begin{abstract} We combine transcendental methods and the modular approaches to the $ABC$ conjecture to show that the largest prime factor of $n^2+1$ is at least of size $(\log_2 n)^2/\log_3n$ where $\log_k$ is the $k$-th iterate of the logarithm. This gives a substantial improvement on the best available estimates, which are essentially of size $\log_2 n$ going back to work of Chowla in 1934. Using the same ideas, we also obtain significant progress on subexpoential bounds for the $ABC$ conjecture, which in a case gives the first improvement on a result by Stewart and Yu dating back over two decades. Central to our approach is the connection between Shimura curves and the $ABC$ conjecture developed by the author.
\end{abstract}

\maketitle



\section{Introduction} For a non-zero integer $n$ let $\Pcal(n)$ be the largest prime factor of $n$, with $\Pcal(\pm 1)=1$. It is a classical problem to give lower bounds for $\Pcal(f(n))$ where $f$ is a non-linear integer polynomial. In 1934 Chowla \cite{Chowla} proved that there is a positive constant $\kappa$ such that
\begin{equation}\label{EqnPn21}
\Pcal(n^2+1)\ge \kappa \cdot \log_2 n
\end{equation}
as $n$ grows, where $\log_k$ is the $k$-th iterate of the logarithm (for an iterated logarithm, we always assume that the argument is large enough for it to be defined.) Since then, this result has been generalized to all polynomials (see \cite{ShoreyTijdeman} and the references therein) but after 90 years only minor improvements on Chowla's theorem are available: to the best of the author's knowledge, the sharpest available estimate is obtained by the theory of linear forms in logarithms and it takes the form
$$
\Pcal(n^2+1)\ge \kappa \cdot \frac{\log_3 n}{\log_4 n}\cdot \log_2 n,
$$
see \cite{ShoreyTijdeman}. We prove a lower bound for $\Pcal(n^2+1)$ which is nearly the square of the previous bounds:

\begin{theorem}\label{Thm1} There is a constant $\kappa>0$ such that as $n$ grows we have
$$
\Pcal(n^2+1)\ge \kappa \cdot\frac{(\log_2 n)^2}{\log_3 n}.
$$
\end{theorem}

For a positive integer $n$ let $\rad(n)$ be its radical, that is, the largest squarefree divisor of $n$. The previous result is in fact a direct consequence of the next theorem:

\begin{theorem}\label{Thm2} There is a constant $\kappa>0$ such that as $n$ grows we have
$$
\rad(n^2+1)\ge \exp\left(\kappa\cdot \frac{(\log_2 n)^2}{\log_3 n}\right).
$$
\end{theorem}

The proof of the previous theorems combines the theory of linear forms in logarithms with results from a modular approach to the $ABC$ conjecture developed by the author using the theory of Shimura curves \cite{PastenShimura}. In particular, the methods pertain to both transcendental number theory and arithmetic geometry. 

More precisely, when we apply linear forms in logarithms to the problem we will separate those primes with large exponent in the factorization of $n^2+1$ from those with small exponent. Then, for each $n$, we construct an elliptic curve and we apply our results from \cite{PastenShimura} to this elliptic curve in order to give a good bound for the number of prime divisors of $n^2+1$ with large exponent. After this point we return to the bounds provided by linear forms in logarithms using this new input to conclude.

The technique developed in this work is not limited to the previous two theorems. In fact, another application of our methods is that we can give improvements on the sharpest available subexponential bounds for the $ABC$ conjecture. Let us recall the statement of the problem.
\begin{conjecture}[The Masser--Oesterl\'e $ABC$ conjecture] Let $\epsilon>0$. There is a number $\kappa_\epsilon>0$ depending only on $\epsilon$ such that the following holds:

Given $a,b,c$ coprime positive integers with $a+b=c$, we have $c\le \kappa_\epsilon \cdot \rad(abc)^{1+\epsilon}$.
\end{conjecture}

In what follows, let us write $R=\rad(abc)$ where $a,b,c$ are triples as in the statement of the $ABC$ conjecture and the term ``absolute constant'' refers to a number independent of all parameters. At present, the best unconditional bound is due to Stewart--Yu \cite{ABC3} and it is of the form
$$
\log c \le \kappa\cdot R^{1/3}(\log R)^3
$$
for certain absolute constant $\kappa$. See also \cite{ABC1,ABC2,MurtyPasten} for other unconditional bounds. While all these bounds are exponential on $R$,  under certain circumstances one can do better obtaining subexponential bounds:

\begin{itemize}
\item[(i)] (See \cite{PastenTruncated} by the author.) Let $\epsilon>0$. There is a number $\kappa_\epsilon>0$ depending only on $\epsilon$ such that if $a \le c^{1-\eta}$ for some number $\eta>0$, then
$$
\log c \le \eta^{-1}\cdot \kappa_\epsilon\cdot \exp\left( (1+\epsilon)\cdot \frac{\log_3R}{\log_2R}\cdot \log R\right).
$$

\item[(ii)] (See \cite{ABC3} by Stewart and Yu.) Let $q=\min\{\Pcal(a),\Pcal(b),\Pcal(c)\}$. Then for certain absolute constant $\kappa>0$ we have
$$
\log c \le q\cdot \exp\left(\kappa \cdot \frac{\log_3R}{\log_2R}\cdot \log R\right).
$$

\end{itemize}

Our method gives a substantial improvement on both  bounds:

\begin{theorem}\label{Thm3} Let $a,b,c$ vary over triples of coprime positive integers with $a+b=c$ and write $R=\rad(abc)$. Then we have the following bounds:
\begin{itemize}

\item[(1)] There is an absolute constant $\kappa>0$ such that if $a \le c^{1-\eta}$ for a number $\eta>0$, then
$$
\log c \le \eta^{-1}\exp\left(\kappa \cdot \sqrt{(\log R)\log_2R}\right).
$$

\item[(2)] Let $q=\min\{\Pcal(a),\Pcal(b),\Pcal(c)\}$. There is an absolute constant $\kappa>0$ for which we have
$$
\log c \le q\cdot \exp\left(\kappa \cdot \sqrt{(\log R)\log_2R}\right).
$$
\end{itemize}
\end{theorem}

It is worth pointing out that item (2) of the previous theorem is the first improvement on Theorem 2 from \cite{ABC3} in more than two decades. 

Using item (2) of Theorem \ref{Thm3} we also get the following improvement on the bound (7) of \cite{ABC3}:
\begin{corollary}\label{Coro4} There is an absolute  constant $\kappa>0$ such that as $x<y$ vary over coprime positive integers, we have
$$
\Pcal(xy(x+y))\ge \kappa \cdot \frac{(\log_2 y)^2}{\log_3 y}.
$$
\end{corollary}
Namely, the bound (7) in \cite{ABC3} is
$$
\Pcal(xy(x+y))\ge \kappa \cdot \frac{(\log_2 y)\log_3 y}{\log_4 y}
$$
which in turn is an improvement of  the earlier bound 
$$
\Pcal(xy(x+y))\ge \kappa \cdot \log_2 y
$$
by van der Poorten, Schinzel, Shorey, and Tijdeman \cite{PSST}.

As the reader will see, with some bookkeeping one can get explicit values for $\kappa$ in Theorems \ref{Thm1}, \ref{Thm2}, and \ref{Thm3} as well as Corollary \ref{Coro4} that work for large enough values of the variables. We leave this task to the interested reader. 

\section{Preliminaries}

\subsection{Bounds coming from linear forms in logarithms} We need estimates for approximation by finitely generated multiplicative groups due to Evertse and Gy\"ory (cf. Theorem 4.2.1 in \cite{EvertseGyory}) that come from the theory of linear forms in logarithms and geometry of numbers. 

Let us first introduce the notation. Let $k$ be a number field of degree $d$ over $\Q$. If $v$ is an archimedian place of $k$ associated to an embedding $\sigma: k\to \C$ (it could be real, or it could come in a complex conjugate pair), we define the $v$-adic norm on $k$
$$
|x|_v = |\sigma(x)|^{\epsilon_v}
$$
where $|-|$ is the usual complex absolute value and $\epsilon_v=1$ if $\sigma$ is real, and $\epsilon_v=2$ if $\sigma$ is complex. On the other hand, if $v$ is a non-archimedian place of $k$ associated to a prime ideal $\pfrak$ of $O_k$, we let $\nu_\pfrak$ be the $\pfrak$-adic valuation on $k$ and define the $v$-adic norm
$$
|x|_v = \Norm(\pfrak)^{-\nu_\pfrak(x)}.
$$
In the case of $k=\Q$ we simply write $\nu_p$ for the $p$-adic valuation when $p$ is a prime number.

The height on $k$ is defined by
$$
h(x)=\frac{1}{d}\sum_v \log \max\{1,|x|_v\}.
$$
Let $\Gamma$ be a finitely generated multiplicative subgroup of $k^*$ and let $\{\xi_1,...,\xi_m\}\subseteq \Gamma$ be a system of generators for $\Gamma/\Gamma_{\rm tor}$ with $m\ge 1$. With these notation and assumptions, from Theorem 4.2.1 in \cite{EvertseGyory} we get:
\begin{theorem}[Approximation bound] \label{ThmLFL} There is a number $K_d$ depending only on $d$ such that the following holds:
\begin{itemize}
\item[(i)] (Archimedian bound) Let $v$ be an archimedian place of $k$. For every $\xi\in \Gamma$ different from $1$ we have
$$
-\log |1-\xi|_v < K_d^{m}\cdot \left(\log \max\{e, h(\xi)\}\right)\prod_{j=1}^mh(\xi_j).
$$
\item[(ii)] (Non-archimedian bound) Let $v$ be a non-archimedian place of $k$ associated with a prime ideal $\pfrak$ of $O_k$. For every $\xi\in \Gamma$ different from $1$ we have
$$
-\log |1-\xi|_v < K_d^{m}\cdot\frac{\Norm(\pfrak)}{\log \Norm(\pfrak)} \left(\log \max\{e,  \Norm(\pfrak)h(\xi)\}\right)\prod_{j=1}^mh(\xi_j).
$$
\end{itemize}
\end{theorem}

\subsection{Bounds coming from the classical modular approach to Szpiro's conjecture}

In \cite{MurtyPasten} Murty and the author showed an unconditional partial result for Szpiro's conjecture using classical modular forms.

\begin{theorem}[Szpiro type bound, \cite{MurtyPasten}] There is an absolute  constant $\kappa>0$ such that for all elliptic curves $E$ over $\Q$ one has
$$
\log \Delta \le \kappa \cdot N\log N
$$
where $\Delta$ and $N$ are the minimal discriminant and the conductor of $E$.
\end{theorem}

The constant $\kappa$ was made explicit in \cite{MurtyPasten} and the result is strong enough to be useful in explicit computations with Diophantine equations; see \cite{MurtyPasten} for the $S$-unit equation and \cite{vKM} for other applications. See \cite{PastenShimura} for improvements.

From the previous theorem one in particular gets:

\begin{corollary}[Bounds for exponents of the minimal discriminant]\label{CoroExponentsDelta} There is an absolute constant $\kappa>0$ such that for all elliptic curves $E$ over $\Q$ and all primes $p$ one has
$$
\nu_p( \Delta) \le \kappa \cdot N\log N
$$
where $\Delta$ and $N$ are the minimal discriminant and the conductor of $E$.
\end{corollary}

\subsection{Bounds coming from Shimura curves}

In \cite{PastenShimura} the author developed a theory based on Shimura curve parametrizations of elliptic curves in order to obtain a new type of unconditional bounds for the $ABC$ and Szpiro's conjecture. Let us state the results that we need.
\begin{theorem}[Bound for elliptic curves, Corollary 16.3 in \cite{PastenShimura}]\label{ThmShimuraE} Let $S$ be a finite set of primes and let $\epsilon>0$. There is a number $\kappa_{S,\epsilon}$ depending only on $S$ and $\epsilon$ such that the following holds:

Let $E$ be an elliptic curve over $\Q$, semistable outside of $S$, with minimal discriminant $\Delta$ and conductor $N$. Then
$$
\prod_{p|N^*} \nu_p(\Delta) \le \kappa_{S,\epsilon}\cdot N^{11/2 + \epsilon}
$$
where $N^*$ is the product of all the primes dividing $N$ not in $S$. 
\end{theorem}

\begin{theorem}[Bound for $ABC$ triples, Theorem 16.8 in \cite{PastenShimura}]\label{ThmShimuraABC} Let $\epsilon>0$. There is a number $\kappa_\epsilon$ depending only on $\epsilon$ such that the following holds:

For all coprime positive integers $a,b,c$ with $a+b=c$ we have
$$
\prod_{p|abc} \nu_p(abc)\le \kappa_\epsilon \cdot \rad(abc)^{8/3+\epsilon}.
$$
\end{theorem}

For the convenience of the reader let us briefly sketch the main ideas in the proof of the previous two results.

Let $E$ be an elliptic curve defined over $\Q$. By the modularity theorem \cite{Wiles, TaylorWiles, BCDT} there is a modular parametrization $\varphi:X_0(N)\to E$. By the Jacquet--Langlands correspondence, for each admissible factorization $N=DM$ there is a Shimura curve parametrization $\varphi_{DM}: X_0^D(M)\to E$ where in particular $X_0(N)=X_0^1(N)$ and $\varphi=\varphi_{D,M}$. We assume that these parametrizations have minimal degree. 

The starting point is to prove a generalization of the Ribet--Takahashi formula \cite{RibetTakahashi} to obtain the  formula
$$
\prod_{p|D} \nu_p(\Delta) =\gamma_{D,M}\cdot  \frac{\deg\varphi}{\deg\varphi_{D,M}}
$$
with an error factor $\gamma_{D,M}$ of controlled height (in the worst case, $h(\gamma_{D,M})\le (1+\epsilon)\log D$), where $\Delta$ is the minimal discriminant of $E$. It is important that this formula is \emph{global}: the contribution of every prime is taken into account. 

 If $h(E)$ denotes the Faltings height of $E$ and $c$ is the Manin constant of $\varphi$, then pulling back a N\'eron differential of $E$ via the parametrizations of $E$ one deduces
$$
\frac{\deg\varphi}{\deg\varphi_{D,M}} = \frac{c^2\|f\|^2e^{2h(E)}}{\|f_{D,M}\|_2^2e^{2h(E)}} = \frac{c^2\|f\|^2}{\|f_{D,M}\|_2^2} 
$$
where $\|-\|_2$ is the Petersson norm, $f$ is the Fourier normalized newform attached to $E$, and $f_{D,M}$ is a quaternionic modular form defined on an integral model of $X_0^{D}(M)$ which is Jacquet--Langalands correspondent associated to $f$.

An important part of the work is to prove a uniform upper bound for the Manin constant $c$ (fixing the set of primes of additive reduction). On the other hand, the upper bound $\|f\|_2^2\ll_\epsilon N^{1+\epsilon}$ is known \cite{MurtyBounds}. Since $f_{D,M}$ extends to an integral model of $X_0^D(M)$ one can use Arakelov theory to give a polynomial lower bound $\|f_{D,M}\|_2^2\gg_\epsilon N^{-(5/3+\epsilon)}M^{-1}$. For this it is crucial to have bounds for the Arakelov height of Heegner points in terms of $L$-functions (extensions of the Chowla--Selberg formula due to Yuan--Zhang \cite{YuanZhang} in the context of Colmez's conjecture) and suitable zero-free regions for the relevant $L$-functions. Putting all together one finally arrives to
$$
\prod_{p|D} \nu_p(\Delta)\ll_\epsilon N^{8/3+\epsilon}DM= N^{11/3+\epsilon}.
$$
Theorem \ref{ThmShimuraE} follows by varying the choice of $M$. Theorem \ref{ThmShimuraABC} follows by choosing $E$ as a Frey--Hellegourach elliptic curve in which case one shows the stronger bound $h(\gamma_{D,M})\le \epsilon \log D$.


\section{The largest prime factor of $n^2+1$} 

The following simple observation will be used a couple of times.

\begin{lemma}\label{LemmaCalculus} Consider a number $A>e$. The real function $t\mapsto t\log (A/t)$ is increasing in the range $1\le t\le A/e$.
\end{lemma}

The next lemma does not give the best bound that the method allows, but it is enough for our purposes.

\begin{lemma}\label{LemmaProd} There is an absolute  constant $K>0$ such that for all positive integers $n$ we have
$$
\prod_{p|n^2+1} \nu_p(n^2+1) \le K\cdot \rad(n^2+1)^{8}.
$$
\end{lemma}
\begin{proof} Let $n$ be a positive integer and consider the elliptic curve
$$
E:\quad y^2 = x^3+3x+2n.
$$
Let $\Delta$ and $N$ be the minimal discriminant and the conductor of $E$. This Weierstrass equation is minimal except perhaps at $2$ and $3$ and it has (not necessarily minimal) discriminant  
$$
-16(4\cdot 3^3 + 27(2n)^2)= -1728(n^2+1).
$$
One checks that $E$ has multiplicative reduction away from $2$ and $3$ (thus, $N$ is squarefree except for uniformly bounded powers of $2$ and $3$) and that the minimal discriminant is
$$
\Delta=-2^s\cdot 3^t\cdot (n^2+1)
$$
where $s$ and $t$ are integers of uniformly bounded absolute value. 

By Corollary \ref{CoroExponentsDelta} there is an absolute constant $\kappa>0$ such that
$$
\nu_2(n^2+1)\nu_3(n^2+1)\le \kappa \cdot N^2(\log N)^2.
$$
Letting $\epsilon>0$ and choosing $S=\{2,3\}$ we apply Theorem \ref{ThmShimuraE} we get a number $\kappa_{S,\epsilon}$ depending only on $S$ and $\epsilon$ such that
$$
\prod_{p|N^*} \nu_p(n^2+1)\le \kappa_{S,\epsilon}\cdot N^{11/2+\epsilon}
$$
where $N^*$ is the product of the primes dividing $N$ other than $2$ and $3$. For a prime $p\ne 2,3$ we have that $p$ divides $N$ if and only if it divides $\Delta$, hence, if and only if it divides $n^2+1$. It follows that
$$
\prod_{p|n^2+1} \nu_p(n^2+1) \le \kappa\cdot \kappa_{S,\epsilon}\cdot N^{15/2+\epsilon}.
$$
As $S$ is fixed, one can choose $\epsilon=1/2$ to obtain
$$
\prod_{p|n^2+1} \nu_p(n^2+1) \le K\cdot \rad(n^2+1)^{8}
$$
for certain absolute constant $K>0$, because $\rad(n^2+1)$ and $N$ agree except, perhaps, by a bounded power of $2$ and $3$ (this is because of the semi-stable reduction at primes $p\ne 2,3$).
\end{proof}

\begin{proof}[Proof of Theorem \ref{Thm2}] Let $n$ be a sufficiently large positive integer and write $i=\sqrt{-1}\in \C$. We consider the equation
$$
(n+i)-(n-i)=2i
$$
in $\Z[i]$. This gives the equation 
$$
1-\frac{n-i}{n+i}=\frac{2i}{n+i}
$$
in the quadratic number field $k=\Q(i)$. 

Consider a factorization $n+i=u\cdot \gamma_1^{e_1}\cdots \gamma_r^{e_r}$ with $\gamma_j$ non-associated irreducible elements of $\Z[i]$ and $u\in\{\pm 1,\pm i\}$. Then we have $n-i=\bar{u}\cdot \bar{\gamma}_1^{e_1}\cdots \bar{\gamma}_r^{e_r}$ where the bar denotes complex conjugation. 

Let 
$$
B=\exp\left(\sqrt{(\log R)\log_2 R}\right)
$$
where $R=\rad(n^2+1)$. In what follows we will use the fact that $R$ grows as $n$ grows ---for instance, by Chowla's result, although we don't need a precise rate of growth.

Define $J=\{1,...,r\}$ and let  $I\subseteq J$ be the set of indices such that $e_j>B$. Let $\xi_j=\bar{\gamma_j}/\gamma_j$ for $j\in J$ and let $\xi_0= \prod_{j\in J-I}\xi_j^{e_j}$.  Let $w=\bar{u}/u$. Then we have
$$
\frac{n-i}{n+i}= w\cdot \xi_0\cdot \prod_{j\in I}\xi_j^{e_j}.
$$
Let $I_0=I\cup\{0\}$ and let $\Gamma$ be the subgroup of $k^\times$ generated by $w$ and the $\xi_j$ for $j\in I_0$. Let $m=1+\#I=\#I_0$. Then the elements $\xi_j$ for $j\in I_0$ generate $\Gamma/\Gamma_{\rm tor}$ and we can write
$$
\frac{2i}{n+i} = 1-\xi
$$
where $(n-i)/(n+i)=\xi= w\cdot \xi_0\cdot \prod_{j\in I}\xi_j^{e_j}\in \Gamma$. Item (i) in Theorem \ref{ThmLFL} (with $d=2$) gives an absolute constant $K$ such that
\begin{equation}\label{Eqnn21a}
\log n\le - 2\log \frac{2}{|n+i|} =-\log \left|1-\xi\right|^2\le K^m \cdot (\log \max\{e,h(\xi)\})\prod_{j\in I_0}h(\xi_j)
\end{equation}
where $|-|$ is the usual absolute value on $\C$. Let us estimate the terms on the right of \eqref{Eqnn21a}. First we have
$$
h(\xi)=h\left(\frac{n-i}{n+i}\right)\le \frac{1}{2}\log |n+i|^2=\log|n+i|
$$
so that
\begin{equation}\label{Eqnn21b}
K^m \cdot (\log \max\{e,h(\xi)\}) \le (2K)^m \log_2 n.
\end{equation}
On the other hand $e_j\le B$ for each $j\in J-I$, and recalling that the $\gamma_j$ are non-associated irreducibles we get
$$
h(\xi_0)\le B\cdot h\left(\prod_{j\in J-I}\gamma_j\right)\le \frac{B}{2}\log \prod_{j\in J} \Norm(\gamma_j)\le B\log \prod_{p|n^2+1} p
$$
which gives
\begin{equation}\label{Eqnn21c}
h(\xi_0)\le B\cdot \log R.
\end{equation}

At this point we note from \eqref{Eqnn21a}, \eqref{Eqnn21b}, and \eqref{Eqnn21c} that if $m=1$ (i.e. $I=\emptyset$) then 
$$
\sqrt{\log n}\le \frac{\log n}{\log_2 n}\le 2KB\log R <\exp\left(K'\cdot \sqrt{(\log R)\log_2R}\right)
$$
for a suitable absolute constant $K'>0$, and the result is proved. So we may assume $m\ge 2$.

Let $p_j$ be the prime number below the irreducible $\gamma_j$. Noticing that for $j\in I$ we have $h(\xi_j)\le \log p_j$ we  get
$$
\prod_{j\in I}h(\xi_j)\le \prod_{j\in I}\log p_j\le \left(\frac{\log R}{m-1}\right)^{m-1}
$$
where we used the arithmetic-geometric mean inequality. Putting this together with \eqref{Eqnn21a}, \eqref{Eqnn21b}, and \eqref{Eqnn21c} we deduce
\begin{equation}\label{Eqnn21d}
\sqrt{\log n}\le \frac{\log n}{\log_2 n}\le (2K)^mB(\log R) \left(\frac{\log R}{m-1}\right)^{m-1}= 2K B(\log R)\left(\frac{2K\cdot \log R}{m-1}\right)^{m-1}.
\end{equation}
Next, we note that 
$$
e_j=\nu_{(\gamma_j)}(n+i)\le 2\nu_{p_j}(n^2+1).
$$
 Therefore the condition $e_j>B$ implies $\nu_{p_j}(n^2+1)> B/2$ and the number of indices $j$ satisfying the former condition is $m-1=\#I$. This gives
$$
\prod_{j\in I}\nu_{p_j}(n^2+1)>(B/2)^{m-1}.
$$
On the other hand, by Lemma \ref{LemmaProd} we have
$$
\prod_{p|n^2+1}\nu_p(n^2+1)\le \kappa\cdot R^8
$$ 
for some absolute constant $\kappa$. This yields
$$
m-1<\frac{8\log R + \log \kappa}{\log(B/2)} < \kappa'\cdot \frac{\log R}{\sqrt{(\log R)\log_2R}}=\kappa'\cdot \sqrt{\frac{\log R}{\log_2R}} 
$$
for a suitable absolute constant $\kappa'$. 

Using Lemma \ref{LemmaCalculus} with $A=2K\log R$, and since 
$$
\kappa'\cdot \sqrt{\frac{\log R}{\log_2R}} < \frac{2K}{e}\log R
$$
for $R$ large enough, we deduce
$$
\left(\frac{2K\cdot \log R}{m-1}\right)^{m-1}\le \left(\frac{2K}{\kappa'}\sqrt{(\log R)\log_2 R}\right)^{\kappa'\sqrt{(\log R)/\log_2R}} \le \exp\left(K''\cdot \sqrt{(\log R)\log_2 R}\right)
$$
for a suitable absolute constant $K''$. Using this in \eqref{Eqnn21d} we obtain
$$
\sqrt{\log n}\le 2KB(\log R)B^{K''}.
$$
Thus, for a suitable absolute constant $M>0$ we obtain 
$$
\log n\le \exp\left(M\cdot \sqrt{(\log R)\log_2 R}\right)
$$
and the result follows.
\end{proof}

\begin{proof}[Proof of Theorem \ref{Thm1}] Write $R=\rad(n^2+1)$. By Theorem \ref{Thm2} we have
$$
R \ge \exp\left(\kappa \cdot \frac{(\log_2 n)^2}{\log_3 n}\right).
$$
Write $P=\Pcal(n^2+1)$. By Chebyshev's bound for the function $\theta(x)=\sum_{p\le x}\log p$, we have
$$
R\le \prod_{p\le P} p \le \exp(4P)
$$
which proves the result.
\end{proof}

\section{Subexponential $ABC$, case 1} 

\begin{proof}[Proof of Theorem \ref{Thm3} item (1)] We keep the notation from the statement and assume that $c$ is large enough. Note that $R$ grows as $c$ grows ---for instance, by the finiteness of solutions of the $S$-unit equation. We can write
$$
\frac{a}{c}=1-\xi
$$
where $\xi=b/c$. Let $\xi_1,...,\xi_r$ be the different prime divisors of $bc$ and let $e_j=\nu_{\xi_j}(b/c)$ (possibly negative). Let 
$$
B=\exp\left(\sqrt{(\log R)\log_2R}\right)
$$
and define $J=\{1,2,...,r\}$ and $I=\{j\in J : |e_j|>B \}$. Let $\xi_0=\prod_{j\in J-I}\xi_j^{e_j}$, let $I_0=I\cup\{0\}$, and let $m=\#I_0$. Let $\Gamma$ be the subgroup of $\Q^\times$ generated by $\xi_j$ for $j\in I_0$; in particular, $\xi=b/c\in \Gamma$.

By item (1) in Theorem \ref{ThmLFL} we have
$$
\eta\cdot \log c\le \log(c/a)=-\log\left|1-\xi\right| \le K^m\cdot(\log \max\{e, h(\xi)\})\prod_{j\in I_0}h(\xi_j)
$$
where $|-|$ is the archimedian absolute value on $\Q$ and $K$ is an absolute constant.

We have $h(\xi)=h(b/c)=\log c\le R^{K'}$ for some absolute constant $K'$, by applying any exponential bound for the $ABC$ conjecture (such as the one in \cite{ABC1} which gives $K'=15$.) On the other hand, we have $h(\xi_0)\le B\log R$, so we obtain
\begin{equation}\label{EqnABC1a}
\eta\cdot \log c\le K'\cdot K^mB(\log R)^2\prod_{j\in I}h(\xi_j).
\end{equation}
If $m=1$ we have $I=\emptyset$ thus obtaining
$$
\eta\cdot \log c\le  K'KB(\log R)^2\le \exp\left(2\sqrt{(\log R)\log_2R}\right)
$$
and the result follows. So we may assume that $m\ge 2$. From \eqref{EqnABC1a} we get
\begin{equation}\label{EqnABC1a}
\eta\cdot \log c\le K'\cdot K^mB(\log R)^2\left(\frac{1}{m-1}\log R\right)^{m-1}\le K'KB^2\left(\frac{K}{m-1}\log R\right)^{m-1}
\end{equation}
by applying the arithmetic-geometric mean inequality.

Let us bound $m$. Fixing a small $\epsilon>0$, from Theorem \ref{ThmShimuraABC} we get
$$
B^{m-1}\le \prod_{j\in I}|e_j|\le \prod_{p| R} \nu_p(abc)\le R^{3} 
$$ 
from which it follows that
$$
m-1\le \frac{3\log R}{\log B} = 3\sqrt{\frac{\log R}{\log_2R}}.
$$
From Lemma \ref{LemmaCalculus} we deduce
$$
\left(\frac{K}{m-1}\log R\right)^{m-1}\le \left(\frac{K}{3}\sqrt{(\log R)\log_2 R}\right)^{3\sqrt{(\log R)/\log_2R}}\le B^{K''}
$$
for some absolute constant $K''$. Putting this together with \eqref{EqnABC1a} we get
$$
\eta\cdot \log c\le K'KB^{2+K''}.
$$
The result follows from the definition of $B$.
\end{proof}

\section{Subexponential $ABC$, case 2} 

\begin{proof}[Proof of Theorem \ref{Thm3} item (2)] By Theorem \ref{Thm3} item (1) it suffices to assume $c^{1/2}\le a<b<c$. For the sake of having a more symmetric formulation of the problem, let $x,y,z\in \Z$ be the same numbers $a,b,c$ up to sign and  in some order, such that $x+y+z=0$. We may assume that $q$ divides $x$ and let us write the equation $x+y+z=0$ as
$$
-\frac{x}{z} = 1-\xi
$$
where $\xi=-y/z$. Let $p_0$ be a prime divisor of $x$ (in particular, $p_0\le q$) such that $\nu_{p_0}(x)$ is maximal among the prime divisors of $x$. Then, since $c^{1/2}\le a< b<c$ we see that
$$
\frac{\log c}{2\log R}\le \nu_{p_0}(x)\le 2\nu_{p_0}(x)\log p_0\le -2\log |1-\xi|_{p_0}. 
$$
Let 
$$
B=\exp\left(\sqrt{(\log R)\log_2 R}\right).
$$
As in the proof of Theorem \ref{Thm2} item (1), we may use Theorem \ref{ThmShimuraABC} to bound the number of prime divisors of $xz$ with exponent larger than $B$. Then, using an argument very similar to that in the proof of Theorem \ref{Thm2} item (1), but applying item (ii) of Theorem \ref{ThmLFL} (with $v=p_0$) instead of item (i), we deduce that there is some absolute constant $K$ such that
$$
-\log |1-\xi|_{p_0} \le p_0\cdot B^K\le q \cdot B^K
$$
and the result follows.
\end{proof}

\begin{proof}[Proof of Corollary \ref{Coro4}] We take $a=x$, $b=y$ and $c=x+y$. Since $c<2b=2y$ we may express the desired lower bound for $\Pcal(xy(x+y))$ in terms of $c$.

We may assume $q<(\log_2 c)^2/\log_3 c$ for otherwise the result directly holds. By item (2) of Theorem \ref{Thm3} we get
$$
\log c \le \exp\left(K\cdot \sqrt{(\log R)\log_2 R}\right)
$$
for certain absolute constant $K>0$, where $R=\rad(abc)$. This gives
$$
\log R \ge K'\cdot \frac{(\log_2 c)^2}{\log_3 c}
$$
for certain absolute constant $K'$. By Chebyshev's bound we have $\exp(4\Pcal(abc)) \ge R$ and the result follows. 
\end{proof}


\section{Acknowledgments}

Supported by ANID Fondecyt Regular grant 1230507 from Chile. I thank K\'alm\'an Gy\"ory and Cameron L. Stewart for valuable comments on these results. And I am particularly indebted to Samuel Le Fourn and M. Ram Murty for carefully reading an earlier version of this manuscript and suggesting several changes and corrections.


\end{document}